\documentclass[11pt,leqno]{amsart}

\usepackage{amsmath,amsthm}
\usepackage{latexsym}
\usepackage{amssymb}
\usepackage{mathtools}
\usepackage[parfill]{parskip}
\usepackage{mathrsfs}
\usepackage{enumerate}
\usepackage{booktabs} % For tables
\usepackage{hyperref} % For links in bib
\usepackage{xcolor} % For red text
\usepackage{comment}

\newtheorem{theorem}{Theorem}

\newtheorem{lemma}[theorem]{Lemma}
\newtheorem*{claim*}{Claim}
\newtheorem{proposition}[theorem]{Proposition}

\theoremstyle{definition}
\newtheorem{definition}[theorem]{Definition}

\newtheorem{remark}[theorem]{Remark}

\numberwithin{theorem}{section}

\begin{document}

\bibliographystyle{plain}

\title
{Asymptotically short generalizations of $t$-design curves}
\author {Ayodeji Lindblad}
\address{Department of Mathematics, Massachusetts Institute of Technology, Cambridge, Massachusetts 02139}
\email{ayodeji@mit.edu}
\maketitle

\begin{abstract}
Ehler and Gr\"{o}chenig defined spherical $t$-design curves to be curves whose associated line integrals exactly average all degree at most $t$ polynomials. These authors posed the question of finding spherical $t$-design curves $\gamma_t$ on $S^d$ of asymptotically optimal arc length $\ell(\gamma_t)\asymp t^{d-1}$ as $t\to\infty$. This work investigates analogues of this question for \emph{$\varepsilon_t$-approximate} and \emph{weighted $t$-design curves}, proving existence of such curves on $S^d$ achieving this asymptotic arc length for odd $d\in\Bbb N_+$ in the approximate setting (where $\varepsilon_t\asymp1/t$ as $t\to\infty$) and all $d\in\Bbb N_+$ in the weighted setting (where these curves have weight functions which are strictly positive at all but finitely many points). Formulas for such weighted $t$-design curves for $d\in\{2,3\}$ are presented.
\end{abstract}

\section{Background and main results}
\label{sec:intro}

Spherical $t$-designs were defined by Delsarte, Goethals, and Seidel \cite{Delsarte...77} to be finite subsets of spheres such that any polynomial of degree at most $t$ has the same average on a $t$-design as on the entire sphere. These objects are of interest for providing ``good'' finite approximations of the sphere and for their connections to other areas of combinatorics \cite{BannaiBannai09}, data analysis on the sphere, numerical analysis, and beyond. Of special interest are optimally small $t$-designs, which generally prove difficult to construct. Existence of $t$-designs on $S^d$ for all $t$ and $d$ was proven by Seymour and Zaslavsky \cite[Corollary 1]{SeymourZalavsky84}, but the size of the smallest $t$-design on $S^d$ is open in almost all cases. Delsarte, Goethals, and Seidel provided lower bounds (of asymptotic order $t^d$ as $t\to\infty$) on the sizes of spherical $t$-designs on $S^d$ for all $t,d\in\Bbb N_+$ \cite[Theorems 5.11, 5.12]{Delsarte...77}, but even the asymptotic order of size of an optimally small sequence $(X_t)_{t=1}^\infty$ of $t$-designs on $S^d$ as $t\to\infty$ for $d>1$ was a major open problem for decades. This problem was resolved when Bondarenko, Radchenko, and Viazovska showed that the lower bounds of Delsarte, Goethals, and Seidel are asymptotically optimal up to constants $C_d$ depending only on $d$, proving that for any $t\in\Bbb N_+$ and any integer $C\geq C_dt^d$, there exists a $t$-design on $S^d$ of size $C$ \cite[Theorem 1]{Bondarenko...13}.

Motivated by the applications of curves to experimental design and data analysis on the sphere through techniques such as mobile sampling \cite[Section 1]{EhlerGrochenig23}, Ehler and Gr\"ochenig introduced \emph{spherical $t$-design curves}, curves such that any polynomial of degree at most $t$ has the same average value along the curve as on the entire sphere \cite[Definition 2.1]{EhlerGrochenig23}. For any sequence $(\gamma_t)_{t=1}^\infty$ of $t$-design curves on $S^d$, Ehler and Gr\"ochenig showed that the length $\ell(\gamma_t)$ of $\gamma_t$ satisfies $\ell(\gamma_t)\succsim t^{d-1}$ as $t\to\infty$ \cite[Theorem 1.1]{EhlerGrochenig23} (i.e., the sequence $(\ell(\gamma_t)t^{1-d})_{t=1}^\infty$ is bounded away from 0). These authors called such sequences achieving asymptotic equality $t^{d-1}\succsim\ell(\gamma_t)\succsim t^{d-1}$ in this bound (written as $\ell(\gamma_t)\asymp t^{d-1}$) \emph{asymptotically optimal} and posed the problem of proving existence of asymptotically optimal sequences of $t$-design curves on $S^d$, which they solved for $d=2$ using the $d=2$ case of the aforementioned result \cite[Theorem 1]{Bondarenko...13} of Bondarenko, Radchenko, and Viazovska. The present author solved this problem for $d=3$ \cite[Theorem 1.2]{Lindblad24b}, but it remains open in all other dimensions: the shortest sequences of $t$-design curves known to exist on $S^d$ for $d>3$---which arise from applying techniques \cite[Sections 4-6]{EhlerGrochenig23} of Ehler and Gr\"ochenig to asymptotically optimal $t$-design curves on $S^3$ \cite[Theorem 1.2]{Lindblad24b}---have asymptotic order of length $t^{d(d-1)/2-1}$ (and current work in progress \cite{LindbladWIP} will extend these methods to reduce this asymptotic order of length to $t^{\lfloor d/2\rfloor(\lfloor d/2\rfloor-1)+d-1}$). This motivates the exploration of analogues of this problem concerning classes of curves which could find use in the potential data analysis and experimental design applications for which spherical $t$-design curves were introduced.

Ehler, Gr\"ochenig, and Karner recently introduced a generalization of $t$-design curves called \emph{Marcinkiewicz-Zygmund families}, which are families of curves that satisfy Marcinkiewicz-Zygmund inequalities \cite[Definition 4.1]{Ehler...25}. These authors showed that such families can achieve the asymptotic order of length $t^{d-1}$ as $t\to\infty$ of spherical $t$-design curves on $S^d$ \cite[Theorem 1.1]{Ehler...25}. We now consider further such generalizations, which we refer to as \emph{$\varepsilon$-approximate $t$-design curves} and \emph{weighted $t$-design curves}. The former are curves whose associated line integrals average all degree at most $t$ polynomials on the sphere with error at most $\varepsilon$ times their sup norm, while the latter are curves equipped with real-valued weight functions whose associated weighted line integrals exactly average all degree at most $t$ polynomials on the sphere. The analogous notion of $\varepsilon$-approximate $t$-design sets has been considered on projective tori in work \cite[Section 5, Section C]{Iosue...24} of Iosue, Mooney, Ehrenberg, and Gorshkov and on spheres in work \cite[Subsection 6.1]{Dillon24} of Dillon. Another inequivalent notion of ``approximate'' spherical $t$-design sets has also been investigated \cite{ChenZhou18}, but these objects are specifically weighted $t$-design sets with ``approximately equal'' weights. As opposed to approximate $t$-design sets, weighted $t$-design sets---and, more generally, \emph{quadrature} and \emph{cubature} formulas---have received decades of substantial study \cite{Cools97,Cools03,Gauss1814,Mysovskikh81,Sobolev74, SobolevVaskevitch96,Stroud71,Tchakaloff57}. In fact, notable progress on the problem of proving existence of asymptotically optimal $t$-design sets on spheres of dimension $d>1$ included the work \cite{Mhaskar...01} of Mhaskar, Narcowich, and Ward which showed that there exist \emph{weighted} $t$-design sets $(W_t)_{t=1}^\infty$ on $S^d$ (with almost equal weights) satisfying $|W_t|\asymp t^d$ as $t\to\infty$.

The main results of this manuscript are presented by Theorems \ref{thm:approxmain} and \ref{thm:weight}, which state that there exist $\varepsilon_t$-approximate and weighted $t$-design curves of asymptotic order of arc length $t^{d-1}$ as $t\to\infty$ on all odd-dimensional spheres in the approximate setting (where we have $\varepsilon_t\asymp1/t$ as $t\to\infty$) and on all spheres in the weighted setting (in which case these weighted $t$-design curves have weight functions which are strictly positive at all but finitely many points and have asymptotically optimal arc length among sequences (indexed in $t$) of such curves).

\begin{theorem}[Main theorem on approximate $t$-design curves]\label{thm:approxmain}
For any $d\geq3$, there exists a sequence $(\gamma_t)_{t=1}^\infty$ of $\varepsilon_t$-approximate $t$-design curves on $S^{d}$ such that, as $t\to\infty$, $\ell(\gamma_t)\asymp t^{d-1}$ for $d$ odd, $\ell(\gamma_t)\asymp t^{2d-3}$ for $d$ even, and $\varepsilon_t\asymp1/t$.
\end{theorem}

We formally introduce approximate $t$-design curves and present basic results concerning these objects in Section \ref{sec:approxprelim}, then formalize a construction (communicated by Proposition \ref{pro:const}) of approximate $t$-design curves on $S^{2n+1}$ from $\lfloor t/2\rfloor$-design sets on $\Bbb{CP}^n$ in Section \ref{sec:const}. In Section \ref{sec:asymp}, we first prove the cases of Theorem \ref{thm:approxmain} when $d$ is odd by combining this construction with asymptotic bounds \cite{Breger...18,Etayo...18} concerning $t$-design sets on the complex projective space $\Bbb{CP}^n$ as $t\to\infty$. We then prove the cases of Theorem \ref{thm:approxmain} when $d$ is even from the cases of the theorem when $d$ is odd by applying the natural approximate analogue of a construction \cite[Sections 4-6]{EhlerGrochenig23} of Ehler and Gr\"ochenig.

%Note that Theorem \ref{thm:approxmain} and Proposition \ref{pro:const} hold (with slight changes to constants in Proposition \ref{pro:const} by a factor of (about) 2) when $S^d$ is replaced with $d$-dimensional real projective space $\Bbb{RP}^d$.
%We do not explicitly treat the real projective case to slightly simplify the proofs below, but the theorems follow from exactly the same methods described.

\begin{theorem}[Main theorem on weighted $t$-design curves]\label{thm:weight}
For any $d\geq2$, there exists a constant $W_d>0$ such that for any $t\in\Bbb N_+$ and $W\geq W_dt^{d-1}$, there exists a weighted $t$-design curve on $S^d$ of length $W$ with weight function which is strictly positive at all but finitely many points. For example, fix any $t\in\Bbb N_+$. Writing
\begin{equation}\label{eq:explicitweight2}
\begin{gathered}
\alpha=(\alpha_1,\alpha_2,\alpha_3):[0,1]\to S^2, \\
\alpha_1(s)=\sqrt{1-\alpha_3(s)^2}\cos\left(\frac{\pi\lfloor2ts\rfloor}{t}+\frac{\pi}{2t}\right), \\
\alpha_2(s)=\sqrt{1-\alpha_3(s)^2}\sin\left(\frac{\pi\lfloor 2ts\rfloor}t+\frac\pi{2t}\right), \\
\alpha_3(s)=(-1)^{\lfloor2ts\rfloor}(4ts-2\lfloor2ts\rfloor-1),
\end{gathered}
\end{equation}
$\alpha$ (paired with the weight function $1/|\alpha'|$) will be a weighted $(2t-1)$-design curve on $S^2$ of length $2\pi t$. For any continuous map $\theta:[0,1]\to\Bbb R$ which is smooth on the complement of a finite subset of $[0,1]$, whose derivative is $L^1$-integrable, and which satisfies $\theta(1)-\theta(0)\in2\pi\Bbb Z$, writing
\begin{equation}\label{eq:explicitweight3}
\begin{gathered}
\gamma:[0,1]\to S^3\subset\Bbb C^2, \\
s\mapsto\frac1{\sqrt2}\left(\sqrt{1+\alpha_{1}(r)},\frac {\alpha_2(r)-i\alpha_3(r)}{\sqrt{1+\alpha_{1}(r)}}\right)e^{2\pi is(2t+1)+i\theta(r)}
\end{gathered}
\end{equation}
for $r:=4ts-\lfloor 4ts\rfloor$, $\gamma$ (paired with the weight function $1/|\gamma'|$) will be a weighted $(4t-1)$-design curve on $S^3\subset\Bbb C^2$ of length at least $2\pi\sqrt{4t^4+1}$ (with equality for some $\theta$).
\end{theorem}

After formally introducing weighted $t$-design curves in Section \ref{sec:weight}, we present a construction (communicated by Proposition \ref{pro:weightconst}) of weighted $t$-design curves on $S^d$ from weighted $t$-design sets on $S^{d-1}$ in Subsection \ref{sub:buildweight}. In Subsection \ref{sub:lifting}, we then present the natural weighted analogue of the construction \cite[Theorem 1.3]{Lindblad24b} of $t$-design curves on $S^3$ from $\lfloor t/2\rfloor$-design curves on $S^2$ formalized by the present author in the unweighted setting. In Subsection \ref{sub:weightasymp}, we combine these constructions with the fact that the vertices of a $(t+1)$-gon give a $t$-design on $S^1$ to verify the cases of Theorem \ref{thm:weight} when $d\in\{2,3\}$. We then combine the former construction with an existence result \cite[Theorem 1]{Bondarenko...13} for $t$-design sets on $S^{d-1}$ of asymptotic order of size $t^{d-1}$ as $t\to\infty$ to verify the theorem for general $d\in\Bbb N_+$. Note that Proposition \ref{pro:weightEG} (a direct analogue of a result \cite[Theorem 1.1]{EhlerGrochenig23} of Ehler and Gr\"ochenig) observes that the curves of interest in the theorem (for $W=W_dt^{d-1}$) will be asymptotically optimal, in the sense that any sequence $(\gamma_t,w_t)_{t=1}^\infty$ of weighted $t$-design curves on $S^d$ with non-negative weights $w_t\geq0$ must satisfy $\ell(\gamma_t)\succsim t^{d-1}$ as $t\to\infty$. After noting this, we conclude by combining the former construction with work \cite[Theorem 1.6]{Dillon24} of Dillon to present a result (communicated by Proposition \ref{pro:weightstrength}) analogous to Theorem \ref{thm:weight} concerning asymptotic lengths of weighted $t$-design curves on $S^d$ as $d\to\infty$ (rather than as $t\to\infty$).

%\textcolor{red}{Gr(2,4) curves - could be nice to write them out explicitly here, but not certain that's best. Maybe just do it later? Or just note it can be done...probably best to do it! Note you can use Gauss–Legendre quadrature on $[-1,1]$, these are computable but don't have a closed form. Maybe also do Clenshaw–Curtis, they can be written explicitly}

\section{Approximate $t$-design curves}
\label{sec:approxprelim}

Take $P_t(S^d)$ to be the space of restrictions to $S^d$ of real-valued polynomials on $\Bbb R^{d+1}$ of degree at most $t$, where such a polynomial is an element of the span over $\Bbb R$ of products of at most $t$ coordinate functions $\Bbb R^{d+1}\to\Bbb R$. Also take $\sigma$ to be the uniform measure (meaning a measure on a metric space with respect to which any two balls of the same radius have the same measure) on $S^d$, normalized so that $\sigma(S^d)=1$. We now formally introduce approximate $t$-design curves:

\begin{definition}\label{def:design}
Consider $d\in\Bbb N_+:=\{1,2,\dots\}$ and $t\in\Bbb N:=\{0,1,\dots\}$. For any $\varepsilon\geq0$ and $c>0$, we say that a continuous, piecewise smooth, closed curve $\gamma:[0,1]\to S^d$ with finitely many self-intersections is an \emph{$(\varepsilon,c)$-approximate $t$-design curve} if 
\[\left|c\int_\gamma f-\int_{S^d}f\,d\sigma\right|\leq\varepsilon\sup_{S^d}|f|\]
for all $f\in P_t(S^d)$, where
\[\int_\gamma f:=\int_0^1f(\gamma(s))|\gamma'(s)|\,ds.\]
If $c=1/\ell(\gamma)$ (where $\ell(\gamma):=\int_\gamma1$ is the length of $\gamma$), we say that $\gamma$ is an \emph{$\varepsilon$-approximate $t$-design curve}.
\end{definition}

We now prove basic facts about $(\varepsilon,c)$-approximate $t$-design curves. To this end, fix any $d\in\Bbb N_+$, $t\in\Bbb N$, $\varepsilon\geq0$, and $c>0$.
The first observation we find it natural to make about $(\varepsilon,c)$-approximate $t$-design curves is that any continuous, piecewise smooth curve $\gamma:[0,1]\to S^d$ is a $(c\ell(\gamma)+1,c)$-approximate $t$-design curve on $S^d$ for any $c>0$, while 0-approximate $t$-design curves are exactly $t$-design curves as defined by Ehler and Gr\"ochenig \cite[Definition 2.1]{EhlerGrochenig23}. Additionally, we may observe taking $f=1$ in Definition \ref{def:design} that there are no $(\varepsilon,c)$-approximate $t$-design curves $\gamma$ for 
\[c\not\in[(1-\varepsilon)/\ell(\gamma),(1+\varepsilon)/\ell(\gamma)].\]
Expanding on the relationship between the error term $\varepsilon$ and scaling constant $c$ of an $(\varepsilon,c)$-approximate $t$-design curve, we present Remark \ref{rmk:changec}: 

\begin{remark}\label{rmk:changec}
Consider any $\widetilde c>0$ alongside an $(\varepsilon,c)$-approximate $t$-design curve $\gamma$ on $S^d$. For any $f\in P_t(S^d)$, we have
\[\left|\widetilde c\int_{\gamma}f-\int_{S^d}f\,d\sigma\right|\leq\varepsilon\sup_{S^d}|f|+\left|(\widetilde c-c)\int_\gamma f\right|\leq(\varepsilon+|\widetilde c-c|\ell(\gamma))\sup_{S^d}|f|.\]
Thus, $\gamma$ is also an $(\varepsilon+|\widetilde c-c|\ell(\gamma),\widetilde c)$-approximate $t$-design curve.
\end{remark}

It may also be of interest to observe that any $(\varepsilon,c)$-approximate $t$-design curve on $S^d$ may be perturbed to be smooth and, if $d\neq2$, simple, as is communicated by Remark \ref{rmk:smooth}:

\begin{remark}\label{rmk:smooth}
Consider an $(\varepsilon,c)$-approximate $t$-design curve $\gamma$ on $S^d$ (where we here take $d\geq2$). For any $\delta>0$, there exists a smooth $(\varepsilon+2\delta c,c)$-approximate
$t$-design curve $\widetilde\gamma$ on $S^d$ of length $\ell(\widetilde\gamma)=\ell(\gamma)+\delta$. If $d\neq2$, we may take $\widetilde\gamma$ to be simple.
\end{remark}

\iffalse
\begin{proof}
With notation as in the theorem statement, take $S$ to be the (finite) set of $s\in[0,1]$ such that $\gamma(s)$ is not smooth. Consider an open subset $\widetilde S\subset[0,1]$ such that $\ell(\gamma|_S)=\delta/2$ and $S\subset\gamma(S)$. We may then smooth $\gamma$ to any curve $\widetilde\gamma$ which coincides with $\gamma$ on $[0,1]\setminus S$ and satisfies $\ell(\widetilde\gamma|_{S})=\delta/2$. We then see from Lemma \ref{lem:overlap} that $\widetilde\gamma$ satisfies the desired properties. As $(\varepsilon+2\delta c,c)$-approximate $t$-design curves have finitely many self-intersection points, we may perform this same procedure taking $S$ to be these self-intersection points to ensure $\widetilde\gamma$ will be simple.
\end{proof}
\fi

The remark then allows us to take the curves of interest in Theorem \ref{thm:approxmain} (which are, as constructed, simple geodesic cycles, meaning that they are closed curves which connect finitely many points along geodesics) to be smooth if desired. Finally, we present the fact relevant to the proof of this theorem that a continuous, piecewise smooth, closed curve whose image partially coincides with a piecewise smooth topological 1-manifold which exactly averages polynomials of degree at most $t$ on $S^d$ will be an $(\varepsilon,c)$-approximate $t$-design curve for $\varepsilon$ and $c$ which depend on how closely the image of the curve coincides with the manifold.

% Curves which coincide
\begin{lemma}\label{lem:overlap}
Consider a piecewise smooth topological 1-manifold $L\subset S^d$ satisfying
\[\left|c\int_L f\,ds-\int_{S^d}f\,d\sigma\right|\leq\varepsilon\sup_{S^d}|f|\]
and a continuous, piecewise smooth, closed curve $\gamma:[0,1]\to S^d$ with at most finitely many self-intersections, where $s$ is the arc length measure on $L$. Defining 
\[l:=\gamma^{-1}(\gamma([0,1])\cap L),\quad\varepsilon_\gamma:=\varepsilon+c(s(L)+\ell(\gamma)-2\ell(\gamma|_l)),\]
$\gamma$ is an $(\varepsilon_\gamma,c)$-approximate $t$-design curve.
\end{lemma}

\begin{proof}
For any $f\in P_t(S^d)$ and taking $s$ to be the arc length measure on $L$, we have
\begin{equation*}
\begin{split}
\int_L f\,ds&=\int_{\gamma|_l}f+\int_{L\setminus\gamma(l)}f\,ds+\int_{\gamma|_{[0,1]\setminus l}}f-\int_{\gamma|_{[0,1]\setminus l}}f \\
&=\int_\gamma f+\int_{L\setminus\gamma(l)}f\,ds-\int_{\gamma|_{[0,1]\setminus l}}f.
\end{split}
\end{equation*}
Therefore, we can see that
\begin{equation*}
\begin{split}
\varepsilon\sup_{S^d}|f|&\geq\left|c\int_L f\,ds-\int_{S^d}f\,d\sigma\right| \\
&=\left|c\left(\int_\gamma f+\int_{L\setminus\gamma(l)}f\,ds-\int_{\gamma|_{[0,1]\setminus l}}f\right)-\int_{S^d}f\,d\sigma\right| \\
&\geq\left|\left|c\int_\gamma f-\int_{S^d}f\,d\sigma\right|-c\left|\int_{L\setminus\gamma(l)}f\,ds-\int_{\gamma|_{[0,1]\setminus l}}f\right|\right|,
\end{split}
\end{equation*}
so we have 
\begin{equation*}
\begin{split}
\left|c\int_\gamma f-\int_{S^d}f\,d\sigma\right|&\leq \varepsilon\sup_{S^d}|f|+c\left|\int_{L\setminus\gamma(l)}f\,ds-\int_{\gamma|_{[0,1]\setminus l}}f\right| \\
&\leq(\varepsilon + c(s(L)+\ell(\gamma)-2\ell(\gamma|_l)))\sup_{S^d}|f|.
\end{split}
\end{equation*}
This is exactly the desired result.
\end{proof}

\section{Building approximate $t$-design curves}
\label{sec:const}

In this section, we prove Proposition \ref{pro:const}, which allows us to build an approximate $t$-design curve on $S^{2n+1}$ from a $\lfloor t/2\rfloor$-design set on $\Bbb{CP}^n$.

\begin{proposition}\label{pro:const}
Fixing $n\in\Bbb N_+$, $t\geq2$, and a $\lfloor t/2\rfloor$-design set $Y$ on $\Bbb{CP}^n$, consider the constants 
\[W_Y\in\left(0,2(|Y|-1)\sup_{y_1,y_2\in Y}d_{\Bbb{CP}^n}(y_1,y_2)\right],\quad M_Y>0\] 
as in Subsection \ref{sub:connecting}. For any $\delta\in(0,M_Y)$, we may construct a simple $\left(\frac{W_Y+\delta}{2\pi|Y|},\frac1{2\pi|Y|}\right)$-approximate $t$-design curve $\gamma_Y$ on $S^{2n+1}$ of length $\ell(\gamma_Y)=2\pi|Y|+W_Y-\delta$. We may take $\gamma_Y$ to be a geodesic cycle or smooth.
\end{proposition}

%Applying Remark \ref{rmk:changec}, we can see that $\gamma_Y$ as in Proposition \ref{pro:const} will also be a $(W_Y/(\pi|Y|))$-approximate $t$-design curve.
%Note that Theorem \ref{thm:approxmain} and Proposition \ref{pro:const} hold (with slight changes to constants in Proposition \ref{pro:const} by a factor of (about) 2) when $S^d$ is replaced with $d$-dimensional real projective space $\Bbb{RP}^d$.
%We do not explicitly treat the real projective case to slightly simplify the proofs below, but the theorems follow from exactly the same methods described.

%To the end of proving Proposition \ref{pro:const}, 
We introduce complex projective $t$-design sets and discuss how they can be used to average degree at most $2t+1$ spherical polynomials in Subsection \ref{sub:sphavg}, then describe the construction of Proposition \ref{pro:const} and use the understanding of complex projective $t$-design sets discussed in Subsection \ref{sub:sphavg} to verify the theorem in Subsection \ref{sub:connecting}. %As noted in Section \ref{sec:intro}, note this proof applies to show the result of Proposition \ref{pro:const} (with slight modifications to certain constants by a factor of (about) 2) when $S^d$ is replaced by $\Bbb{RP}^d$.

\subsection{Complex projective $t$-design sets}
\label{sub:sphavg}

In this subsection, we first introduce complex projective $t$-design sets, then discuss how these objects can be used to average degree at most $2t+1$ polynomials on spheres. To this end, we define the \emph{complex projective space} $\Bbb{CP}^n$ of complex dimension $n$ to be the quotient of $S^{2n+1}$ by the multiplicative action $a\mapsto az$ of $S^1\subset\Bbb C$ on the $(n+1)$-dimensional complex vector space $\Bbb{C}^{n+1}\supset S^{2n+1}$ and equip $\Bbb{CP}^n$ with its uniform measure $\rho$, normalized so that $\rho(\Bbb{CP}^n)=1$. We also consider the \emph{complex projective map}
\begin{equation}\label{eq:cplxmap}
\Pi:S^{2n+1}\to\Bbb{CP}^n,\quad\omega\mapsto[\omega]:=\{\omega z\:|\:z\in S^1\}
\end{equation}
(which we note pushes forward the uniform spherical measure $\sigma$ to $\rho$) and denote by
\begin{equation}\label{eq:PtCPn}
	P_t(\Bbb{CP}^n):=\{f:\Bbb{CP}^n\to\Bbb R\:|\:\Pi^*f\in P_{2t}(S^{2n+1})\}
\end{equation}
the space of polynomials of degree at most $t$ on $\Bbb{CP}^n$. We then say a finite subset $Y\subset\Bbb{CP}^n$ is a \emph{complex projective $t$-design set} on $\Bbb{CP}^n$ if 
\[\frac1{|Y|}\sum_{y\in Y}g(y)=\int_{\Bbb{CP}^n}g\,d\rho\]
for all $g\in P_t(\Bbb{CP}^n)$. Lemma \ref{lem:cptd} describes how complex projective $t$-design sets average spherical polynomials:

\begin{lemma}\label{lem:cptd}
	Consider a $t$-design set $Y$ on $\Bbb{CP}^n$. For any $f\in P_{2t+1}(S^{2n+1})$, we have
	\[\frac1{2\pi|Y|}\int_{\Pi^{-1}(Y)}f\,ds=\int_{S^{2n+1}}f\,d\sigma,\]
	where $s$ is the arc length measure on $\Pi^{-1}(Y)$.
\end{lemma}

\begin{figure}
\begin{center}
\includegraphics[width=.8\textwidth]
{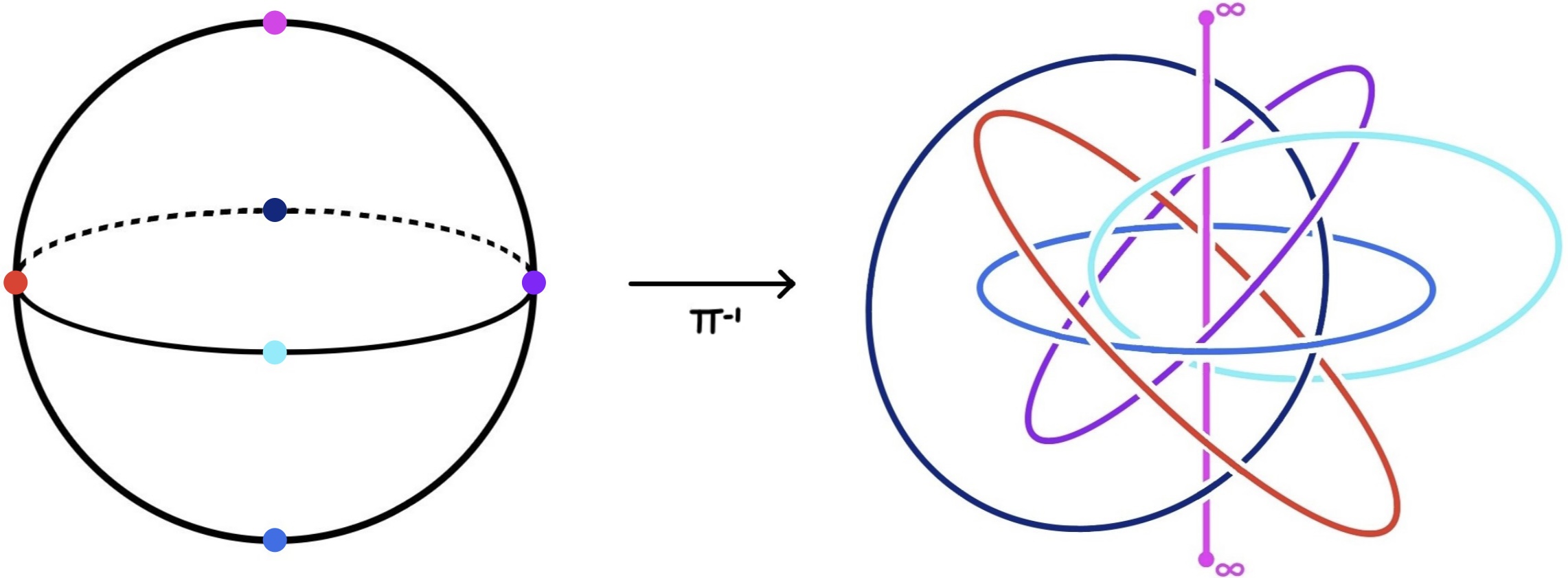}
\caption{\label{fig:TakingPreimage}
The preimage $\Pi^{-1}(O)\subset S^3\cong\Bbb R^3\cup\{\infty\}$ under $\Pi$ of the 3-design set on $\Bbb{CP}^1$ corresponding to the vertices $O$ of an octahedron under the association $\Bbb{CP}^1\cong S^2$ \cite[Example 2.7]{BannaiBannai09}. The average of any degree at most 7 polynomial on $\Pi^{-1}(O)$ equals the average of this polynomial on $S^3$.}
\end{center}
\end{figure}

Lemma \ref{lem:cptd} makes the fundamental observation behind the construction which uses the complex projective map to build a family of $t$-design sets on $S^{2n+1}$ from a $\lfloor t/2\rfloor$-design set on $\Bbb{CP}^n$ introduced by K\"onig \cite[Corollary 1]{Koning98} and also investigated by Kuperberg \cite[Theorem 4.1]{Kuperberg06}. This phenomenon was further investigated by Okuda \cite[Theorem 1.1]{Okuda15}, who---inspired by the independent observation of the construction by Cohn, Conway, Elkies, and Kumar \cite[Section 4]{Cohn...06}---verified a related construction which uses the Hopf map to build a family of $t$-design sets on $S^3$ from a $\lfloor t/2\rfloor$-design set on $S^2\cong\Bbb{CP}^1$. This was generalized in work \cite[Theorem 1.1]{Lindblad26} of the present author, which verifies constructions that build a $t$-design set on $S^{4n+3}$ from a $\lfloor t/2\rfloor$-design set on the quaternionic projective space $\Bbb{HP}^n$ (or $S^4\cong\Bbb{HP}^1$ when $n=1$) and a $t$-design set on $S^3$ and a $t$-design set on $S^{15}$ from a $\lfloor t/2\rfloor$-design set on the octonionic projective line $\Bbb{OP}^1$ (or $S^8\cong\Bbb{OP}^1$) and a $t$-design set on $S^7$. The methods used throughout this work can be applied alongside the result analogous to Lemma \ref{lem:cptd} in these quaternionic and octonionic settings to give rise to an analogue of Proposition \ref{pro:const} where $Y$ is a $\lfloor t/2\rfloor$-design set on $\Bbb{HP}^n$ or $\Bbb{OP}^1\cong S^8$ and $\gamma_Y$ is diffeomorphic to $S^3$ or $S^7$ respectively and approximately averages degree at most $t$ polynomials in a sense similar to that of Definition \ref{def:design}. 

A proof of Lemma \ref{lem:cptd} can be found using the result \cite[Corollary 1]{Koning98} of K\"onig or by applying methods used in that work, the work \cite[Theorem 4.1]{Kuperberg06} of Kuperberg, or that of the present author \cite[Theorem 1.1]{Lindblad26}.

\subsection{Connecting fibers}
\label{sub:connecting}

In this subsection, we prove Proposition \ref{pro:const}. For any finite subset $Y\subset\Bbb{CP}^n$, we first consider a minimum spanning tree $\mathcal T_Y$ (which may be constructed using any of a number of simple algorithms \cite{Milkova07}) for the graph with vertices the elements of $Y$ and one edge for each pair $y_1,y_2\in Y$ such that there exists a minimal geodesic between $y_1$ and $y_2$ whose interior is disjoint from $Y$ which we assign a weighting equal to the length of this minimal geodesic measured with respect to the metric 
\begin{equation}\label{eq:cpndist}
d_{\Bbb{CP}^n}([\omega_1],[\omega_2]):=\inf_{z\in S^1}\arccos\langle\omega_1,z\omega_2\rangle_{\Bbb R}
\end{equation}
on $\Bbb{CP}^n$, where $\langle,\rangle_\Bbb R$ denotes the standard inner product on $\Bbb R^{2n+2}$. We then take $W_Y$ to be double the sum of all weights of edges of $\mathcal T_Y$, $N_Y$ to be the maximal number of edges connected to any vertex of $\mathcal T_Y$, and $M_Y:=2\pi(|Y|-1)/N_Y$.

\begin{figure}
\begin{center}
\includegraphics[width=.72\textwidth]
{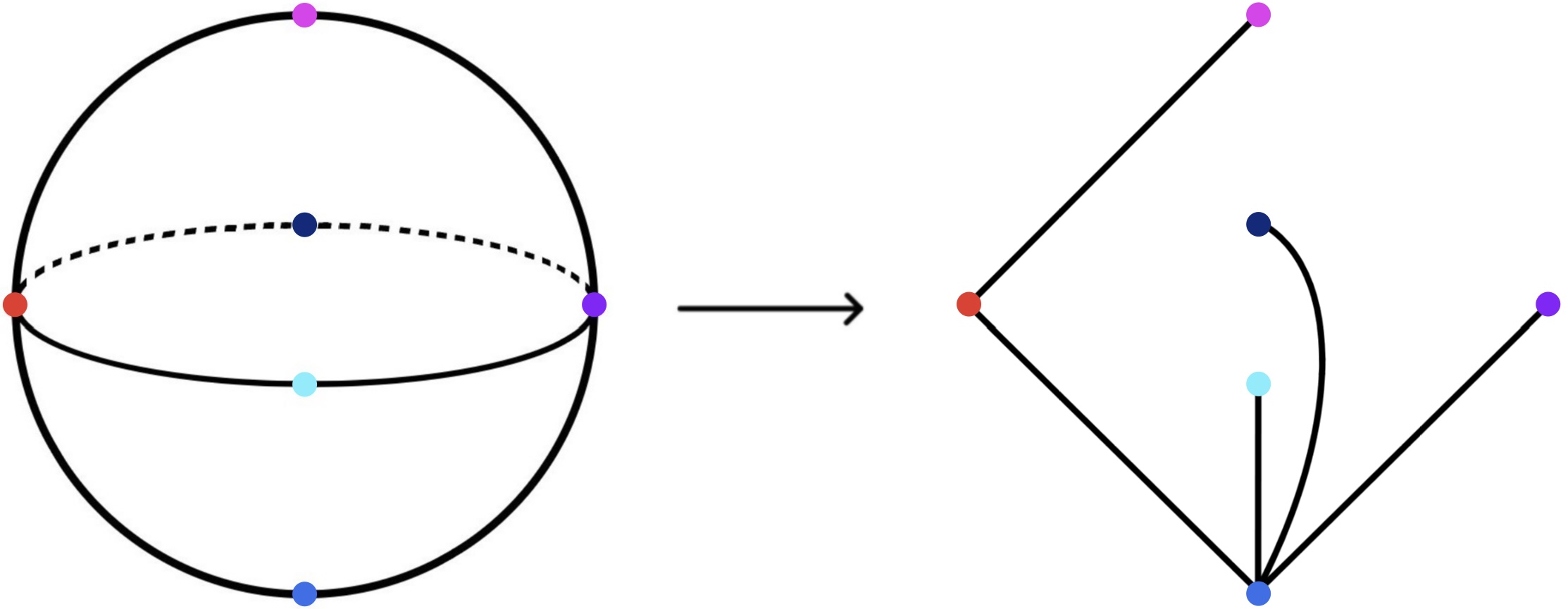}
\caption{\label{fig:MakingGraph}
Making the tree $\mathcal T_O$
as in Section \ref{sec:const}
associated to the vertices $O$ of an octahedron on $\Bbb{CP}^1\cong S^2$.}
\end{center}
\end{figure}

\begin{lemma}\label{lem:connecting}
Consider $n\in\Bbb N_+$, a finite subset $Y\subset\Bbb{CP}^n$, any $\delta\in(0,M_Y)$, and the complex projective map $\Pi:S^{2n+1}\to\Bbb{CP}^n$ as in \eqref{eq:cplxmap}. We may construct a simple geodesic cycle $\gamma_Y$ on $S^{2n+1}$ which, for $s$ the arc length measure, satisfies
\begin{gather*} \label{eq:WYprop}
s(\gamma_Y([0,1])\setminus(\gamma_Y([0,1])\cap\Pi^{-1}(Y)))=W_Y, \\
s(\Pi^{-1}(Y)\setminus(\gamma_Y([0,1])\cap\Pi^{-1}(Y)))=\delta. \label{eq:deltaprop}
\end{gather*}
\end{lemma}

\begin{figure}
\begin{center}
\includegraphics[width=.45\textwidth]
{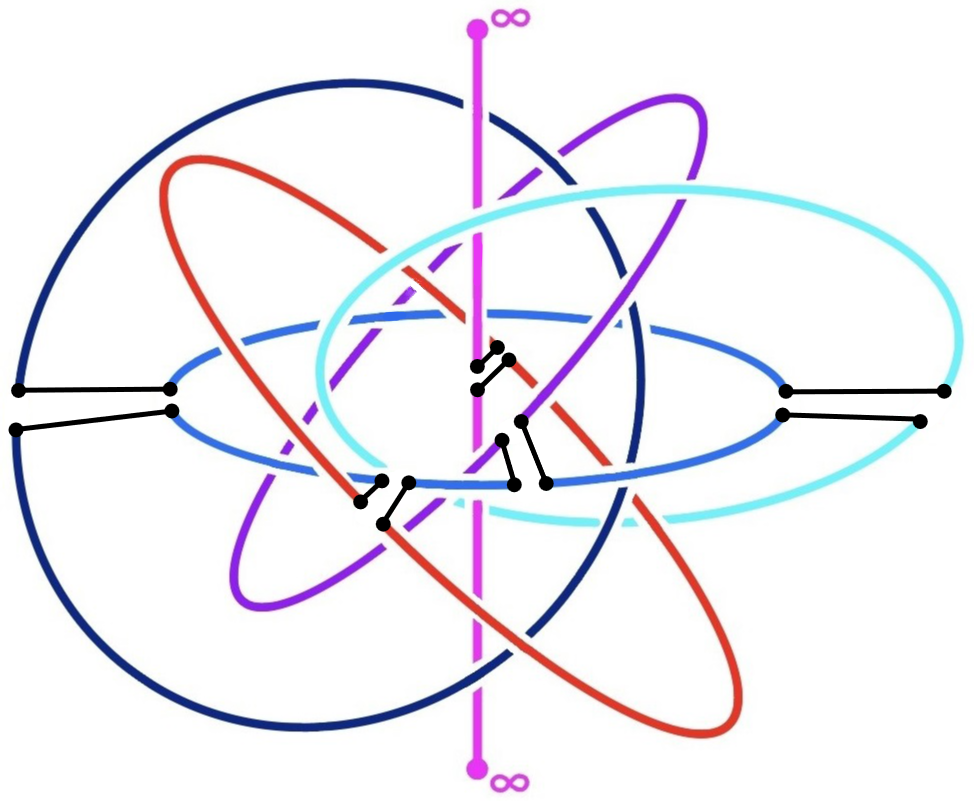}
\caption{\label{fig:AddingLines}
Connecting the disjoint circles comprising $\Pi^{-1}(O)$
to form a path-connected figure $\Gamma_O$
in accordance with the tree $\mathcal T_O$
as in Figure \ref{fig:MakingGraph}. This is the image of a $\left(\frac{5\pi+\delta}{12\pi},\frac{1}{12\pi}\right)$-approximate
7-design curve, with $\delta$ as in Proposition \ref{pro:const}.}
\end{center}
\end{figure}

Applying Lemma \ref{lem:cptd} alongside Lemma \ref{lem:overlap} with $L:=\Pi^{-1}(Y)$ and $\gamma:=\gamma_Y$ for $Y$ a complex projective $\left\lfloor t/2\right\rfloor$-design set, we obtain Proposition \ref{pro:const} when $\gamma_Y$ is a geodesic cycle as a consequence of Lemma \ref{lem:connecting}. Applying Remark \ref{rmk:smooth} and noting that we can make $\delta>0$ arbitrarily small and then perturb $\gamma_Y$ to increase its length and make it smooth, we can see that we may take $\gamma_Y$ to be smooth.

\begin{proof}[Proof of Lemma \ref{lem:connecting}]
Denote by $y_0\in Y$ the root of the tree $\mathcal T_Y$. For each $y\in Y\setminus\{y_0\}$, consider the geodesic $e_y:[0,1]\to\Bbb{CP}^n$ in the set of edges of $\mathcal T_Y$ satisfying $e_y(0)=p$ ($p$ the parent of $y$ in $\mathcal T_Y$) and $e_y(1)=y$. For $z_y\in p$, we then consider a geodesic $g_{z_y}:[0,1]\to S^{2n+1}$ on $S^{2n+1}$  satisfying $g_{z_y}(0)=z_y$ and $\Pi\circ g_{z_y}=e_y$ of minimal length. %\cite[Lemma 2.1]{Lindblad24b}.
Denoting by $\{c_{y,j}\}_{j=1}^{n_y}$ the $n_y\in\{0,\dots,N_Y\}$ children of $y\in Y$ and considering any $\delta_{y}\in\left[0,\frac{2\pi}{n_y+1}-\widetilde\delta\right)$ for $\widetilde\delta:=\delta/(2(|Y|-1))$, we pick any point of $y_0\subset S^{2n+1}$ (which, in an abuse of notation, we call $g_{z_{y_0}}(1)$ for convenience) and iteratively define 
\[z_{c_{y,j}}:=g_{z_y}(1)e^{2\pi ij/(n_y+1)+i\delta_{y}}\in y\subset S^{2n+1}\]
for all $y\in Y$ and $j\in\{1,\dots,n_y\}$. We can observe that the sets 
\[D_y:=\{z_ye^{i\theta}\:|\:\theta\in[0,\widetilde\delta]\}\cup\{g_{z_y}(1)e^{i\theta}\:|\:\theta\in[0,\widetilde\delta]\}\subset\Pi^{-1}(Y)\]
are pairwise disjoint for all $y\in Y\setminus\{y_0\}$.
%\[2\pi/n_y-\delta_{c_{y,j-1}}+\delta_{c_{y,j}}\geq2\pi/n_y-\delta_{c_{y,j-1}}>2\pi/N_Y-2\pi/N_Y+\widetilde\delta=\widetilde\delta.\]
We can also arrange that the sets 
\[G_y:=g_{z_y}([0,1])\cup\left(g_{z_y}([0,1])e^{i\widetilde\delta}\right)\] 
are pairwise disjoint by perturbing $\delta_y,\delta_{\widetilde y}\in[0,2\pi)$ whenever $G_{c_{y,j}}\cap G_{c_{\widetilde y,\widetilde j}}\neq\emptyset$ for some $j\in\{1,\dots,n_y\}$ and $\widetilde j\in\{1,\dots,n_{\widetilde y}\}$. Therefore, noting that $G_y\cap \Pi^{-1}(Y)=\partial G_y=\partial D_y$ for all $y\in Y\setminus\{y_0\}$ and that $\Pi^{-1}(Y)$ and $G_y$ are finite unions of geodesics in $S^{2n+1}$, we can see that the set 
\[\Gamma_Y:=\left(\Pi^{-1}(Y)\setminus\bigcup_{y\in Y\setminus\{y_0\}}D_y\right)\bigcup_{y\in Y\setminus\{y_0\}}G_y\]
is a finite union of geodesics in $S^{2n+1}$. Observing that $\Pi(G_y)=e_y([0,1])$, we can see by construction since $\mathcal T_Y$ is connected that $\Gamma_Y$ is homeomorphic to $S^1$, and picking any continuous, piecewise smooth, simple closed curve $\gamma_Y:[0,1]\to\Gamma_Y$, $\gamma_Y$ will be a geodesic cycle satisfying
\begin{gather*}
s(\gamma_Y([0,1])\setminus(\gamma_Y([0,1])\cap\Pi^{-1}(Y)))=s\left(\bigcup_{y\in Y\setminus\{y_0\}}G_y\right)=W_Y, \\
s(\Pi^{-1}(Y)\setminus(\gamma_Y([0,1])\cap\Pi^{-1}(Y)))=s\left(\bigcup_{y\in Y\setminus\{y_0\}}D_y\right)=\delta
\end{gather*}
as desired.
\end{proof}

\section{Asymptotics of approximate $t$-design curves}
\label{sec:asymp}

In this section, to prove Theorem \ref{thm:approxmain}, we combine Proposition \ref{pro:const} with the following asymptotic result for $t$-design sets on $\Bbb{CP}^n$:

\begin{proposition}[Theorem 2.2 of Etayo, Marzo and Ortega-Cerd\`a \cite{Etayo...18}, Theorem 3.2 of Breger, Ehler, and Gr\"af \cite{Breger...18}]\label{pro:CPnasymp}
For any $n\in\Bbb N_+$, there exists a sequence $(Y_t)_{t=1}^\infty$ of $t$-design sets on $\Bbb{CP}^n$ of asymptotically optimal size $|Y_t|\asymp t^{2n}$, and such a sequence will satisfy
\begin{equation}\label{eq:covrad}
\rho_t:=\sup_{w\in\Bbb{CP}^n}\inf_{y\in Y_t}d_{\Bbb{CP}^n}(w,y)\asymp1/t
\end{equation}
as $t\to\infty$ for $d_{\Bbb{CP}^n}$ the metric on $\Bbb{CP}^n$ as in \eqref{eq:cpndist}.
\end{proposition}

As stated, the works \cite{Etayo...18} of Etayo, Marzo and Ortega-Cerd\`a and \cite{Breger...18} of Breger, Ehler, and Gr\"af prove the results which combine to provide Proposition \ref{pro:CPnasymp} for objects which are defined slightly differently from the notion of $t$-design sets of interest in this work. In fact, these definitions are ultimately equivalent; we note this in Subsection \ref{sub:cplx} after we complete the proof of Theorem \ref{thm:approxmain}. To prove this theorem, we first verify it on all odd-dimensional spheres using Proposition \ref{pro:const}, then apply this result alongside a generalized construction \cite[Sections 4-6]{EhlerGrochenig23} of Ehler and Gr\"ochenig to prove the desired result on all even-dimensional spheres.

\begin{proof}[Proof of Theorem \ref{thm:approxmain} for $d$ odd]
With notation as in the theorem statement, writing $d:=2n+1$, consider a sequence $(Y_{\lfloor t/2\rfloor})_{t=1}^\infty$ of $\lfloor t/2\rfloor$-design sets on $\Bbb{CP}^n$ of asymptotically optimal size $|Y_{\lfloor t/2\rfloor}|\asymp t^{2n}$ as $t\to\infty$, as we know exists from Proposition \ref{pro:CPnasymp}. With notation as in Proposition \ref{pro:const}, we take $W_{t}:=W_{Y_{\lfloor t/2\rfloor}}$ alongside $\delta=\delta_t:=\min(1,W_{t})/t$ and consider the simple $\varepsilon_t$-approximate $t$-design curve $\gamma_t:=\gamma_{Y_{\lfloor t/2\rfloor}}$ of length $\ell(\gamma_t)=2\pi|Y_{\lfloor t/2\rfloor}|+W_{t}-\delta_t$ guaranteed to exist from the proposition paired with Remark \ref{rmk:changec} for $\varepsilon_t:=W_t/(\pi|Y_{\lfloor t/2\rfloor}|)$. Noting from the definition of $W_t$ in Subsection \ref{sub:connecting} and considering the covering radius $\rho_{\lfloor t/2\rfloor}$ of $Y_{\lfloor t/2\rfloor}$ as in \eqref{eq:covrad}, we can see that $W_t\leq 4(|Y_{\lfloor t/2\rfloor}|-1)\rho_{\lfloor t/2\rfloor}$. Therefore, noting from Proposition \ref{pro:CPnasymp} that $\rho_{\lfloor t/2\rfloor}\asymp1/t$ and that $|Y_{\lfloor t/2\rfloor}|\asymp t^{2n}$ as $t\to\infty$, we have that $W_t\lesssim t^{2n-1}$ as $t\to\infty$. Therefore, we have $\delta_t\lesssim1/t$ as $t\to\infty$, so $\ell(\gamma_t)\asymp t^{2n}$ and $\varepsilon_t\precsim t^{2n-1}/t^{2n}=1/t$ as $t\to\infty$.
\end{proof}

Now, consider the result \cite[Sections 4-6]{EhlerGrochenig23} of Ehler and Gr\"ochenig which provides a construction of $t$-design curves on $S^{d}$ from $t$-design curves on $S^{d-1}$, which we note extends to the setting of $\varepsilon_t$-approximate $t$-design curves:

\begin{proposition}[Sections 4-6 of Ehler and Gr\"ochenig \cite{EhlerGrochenig23}]\label{pro:caps}
Fix $d-1\in\Bbb N_+$, $\varepsilon_t\geq0$ for $t\in\Bbb N_+$, and a sequence $(\alpha_t)_{t=1}^\infty$ of $\varepsilon_t$-approximate $t$-design curves on $S^{d-1}$. There exists a sequence $(\gamma_t)_{t=1}^\infty$ of $\varepsilon_t$-approximate $t$-design curves on $S^d$ of asymptotic order $\ell(\gamma_t)\asymp t^{d-1}\ell(\alpha_t)$ of arc length as $t\to\infty$.
\end{proposition}

This extension to $\varepsilon_t$-approximate $t$-design curves follows exactly as in the original proof, so we appeal to that work \cite[Sections 4-6]{EhlerGrochenig23} to prove Proposition \ref{pro:caps}. The even-dimensional settings of Theorem \ref{thm:approxmain} then directly follow from combining Proposition \ref{pro:caps} with the odd-dimensional settings of the theorem, completing the proof altogether.

\subsection{Polynomials on complex projective spaces}
\label{sub:cplx}

Recall we noted that, as stated, the works \cite{Etayo...18} of Etayo, Marzo and Ortega-Cerd\`a and \cite{Breger...18} of Breger, Ehler, and Gr\"af, prove the results whose combination is communicated by Proposition \ref{pro:CPnasymp} for objects which are defined slightly differently from how we define $t$-design sets; for completeness, we now show that the definitions of these objects---which we call $Q_t$-design sets and $R_t$-design sets---are equivalent to the definition of $t$-design sets we use. To this end, consider the embedding
\[\varphi:\Bbb{CP}^n\to \Bbb C^{(n+1)^2}\cong\Bbb R^{2(n+1)^2},\quad[\omega]\mapsto \omega\omega^*\]
($\omega^*$ the conjugate transpose of $\omega$) of $\Bbb{CP}^n$ into Euclidean space taking a point $[\omega]$ to the projection matrix mapping $\Bbb{C}^{n+1}$ to the complex line containing $\omega$. We then define
\begin{equation}\label{eq:Qt}
Q_t(\Bbb{CP}^n):=\varphi^*(P_t(\Bbb R^{2(n+1)^2})).
\end{equation}
Additionally, for $\Delta_{\Bbb{CP}^n}$ the Laplace-Beltrami operator (with non-negative eigenvalues) on $\Bbb{CP}^n$, we denote by $R_t(\Bbb{CP}^n)$ the space of \emph{diffusion polynomials} of degree at most $t$ on $\Bbb{CP}^n$, the span over $\Bbb R$ of real-valued eigenfunctions of $\Delta_{\Bbb{CP}^n}$ in $C^\infty(\Bbb{CP}^n)$ of eigenvalue at most $4t(t+n)$.
Then, we say a finite subset $Y\subset\Bbb{CP}^n$ is a \emph{$Q_t$-design set} or a \emph{$R_t$-design set} if 
\[\frac1{|Y|}\sum_{y\in Y}g(y)=\int_{\Bbb{CP}^n}g\,d\rho\]
for all $g\in Q_t(\Bbb{CP}^n)$ or for all $g\in R_t(\Bbb{CP}^n)$ respectively.

\begin{lemma}\label{lem:poly}
For any $t\in\Bbb N$, we have $P_t(\Bbb{CP}^n)=Q_t(\Bbb{CP}^n)=R_t(\Bbb{CP}^n)$.
\end{lemma}

As a corollary, we see that the conditions that a finite subset of $\Bbb{CP}^n$ is a $t$-design set, a $Q_t$-design set, and a $R_t$-design set are equivalent. %We omit a proof of the lemma for brevity, and note that the fact is straightforward to show.

\begin{proof}[Proof of Lemma \ref{lem:poly}]
We first show that $P_t(\Bbb{CP}^n)=Q_t(\Bbb{CP}^n)$. To this end, note that for any $c\in\Bbb C$, the real and imaginary parts of $c$ respectively satisfy
\begin{equation}\label{eq:realim}
\Re(c)=\frac12(\overline c + c),\quad\Im(c)=\frac12(\overline c-c)i.
\end{equation}
Observe from the definition \eqref{eq:Qt} of $Q_t(\Bbb{CP}^n)$ that $\Pi^*(Q_t(\Bbb{CP}^n))$ is exactly the span over $\Bbb R$ of products of at most $t$ of the functions 
\[\omega\mapsto\Re(\omega_i\overline{\omega_j}),\quad\omega\mapsto\Im(\omega_i\overline{\omega_j})\quad(i,j\in\{1,\dots,n+1\})\]
taking $\omega\in S^{2n+1}\subset\Bbb C^{n+1}$ to the real and imaginary parts of entries of $\omega\omega^*$. Thus, we may see from \eqref{eq:realim} that $\Pi^*(Q_t(\Bbb{CP}^n))\subset P_{2t}(S^{2n+1})$, so $Q_t(\Bbb{CP}^n)\subset P_t(\Bbb{CP}^n)$. To show the reverse inclusion, pick $g\in P_t(\Bbb{CP}^n)$. We see from the definition \eqref{eq:PtCPn} of $P_t(\Bbb{CP}^n)$ that $\Pi^*g$ is in the span over $\Bbb R$ of products of at most $2t$ of the functions 
\[\omega\mapsto\Re(\omega_i),\quad\omega\mapsto\Im(\omega_i)\quad(i\in\{1,\dots,n+1\})\]
taking $\omega\in S^{2n+1}\subset\Bbb C^{n+1}$ to the real and imaginary parts of its entries. 
\eqref{eq:realim} allows us to pick $c_{\alpha,\beta}\in\Bbb C$ for each $(\alpha,\beta)\in I$ such that 
\[(\Pi^*g)(\omega)=\sum_{(\alpha,\beta)\in I}c_{\alpha,\beta}\omega^{\alpha,
\beta}\quad\textit{for}\quad\omega^{\alpha,
\beta}:=\prod_{i=1}^{n+1}\omega_i^{\alpha_i}\overline{\omega_i}^{\beta_i},\]
where we define
\[I:=\left\{(\alpha,\beta)\in\Bbb N^{2n+2}\:\middle|\:\sum_{i=1}^{n+1}(\alpha_i+\beta_i)\leq2t\right\}.\]
Note that $\omega^{\alpha,\beta}$ lies in the complexification of $\Pi^*(Q_t(\Bbb{CP}^n))$ for any
\[(\alpha,\beta)\in I_0:=\left\{(\alpha,\beta)\in I\:\middle|\:\sum_{i=1}^{n+1}(\alpha_i-\beta_i)=0\right\},\]
so if we show that 
\begin{equation}\label{eq:I0}
(\Pi^*g)(\omega)=\sum_{(\alpha,\beta)\in I_0}c_{\alpha,\beta}\omega^{\alpha,
\beta},
\end{equation}
we will have $\Pi^*g\in\Pi^*(Q_t(\Bbb{CP}^n))$. Picking any $z_w=(z_{w,1},\dots,z_{w,n+1})\in w$ for $w\in\Bbb{CP}^n$ and fixing any $(\alpha,\beta)\in I$, we may observe that
\begin{equation*}\label{eq:fiberint}
\begin{split}
\int_w\omega^{\alpha,\beta}\,d\sigma(\omega)&=\int_{S^1}\prod_{i=1}^{n+1}(z_{w,i}\zeta)^{\alpha_i}(\overline{z_{w,i}\zeta})^{\beta_i}\,d\sigma(\zeta) \\
&=\left(\prod_{i=1}^{n+1}z_{w,i}^{\alpha_i}\overline{z_{w,i}}^{\beta_i}\right)\int_{S^1}\prod_{i=1}^{n+1}\zeta^{\alpha_i}\overline\zeta^{\beta_i}\,d\sigma(\zeta)\\
&=\left(\prod_{i=1}^{n+1}z_{w,i}^{\alpha_i}\overline{z_{w,i}}^{\beta_i}\right)\int_{S^1}\zeta^{\sum_{i=1}^{n+1}(\alpha_i-\beta_i)}\,d\sigma(\zeta),
\end{split}
\end{equation*}
which equals 0 whenever $(\alpha,\beta)\not\in I_0$. Picking any $\omega\in w\in\Bbb{CP}^n$ and observing that $\omega^{\alpha,\beta}\in \Pi^*(C^\infty(\Bbb{CP}^n))$ when $(\alpha,\beta)\in I_0$,
% and that
%\[f(\omega)=\int_{w}f\,d\sigma\]
%for any $f\in \Pi^*(C^\infty(\Bbb{CP}^n))$, 
we therefore have that
\begin{equation*}
\begin{split}
(\Pi^*g)(\omega)&=\int_w\Pi^*g\,d\sigma \\
&=\sum_{(\alpha,\beta)\in I}c_{\alpha,\beta}\int_w\omega^{\alpha,\beta}\,d\sigma(\omega) \\
&=\sum_{(\alpha,\beta)\in I_0}c_{\alpha,\beta}\int_w\omega^{\alpha,\beta}\,d\sigma(\omega) \\
&=\sum_{(\alpha,\beta)\in I_0}c_{\alpha,\beta}\omega^{\alpha,\beta},
\end{split}
\end{equation*}
showing \eqref{eq:I0}. Thus, $\Pi^*g\in\Pi^*(Q_t(\Bbb{CP}^n))$, so $P_t(\Bbb{CP}^n)\subset Q_t(\Bbb{CP}^n)$. This completes the proof that $P_t(\Bbb{CP}^n)=Q_t(\Bbb{CP}^n)$.

We now show that $P_t(\Bbb{CP}^n)=R_t(\Bbb{CP}^n)$. To this end, we first note that the Hopf map is a Riemannian submersion with totally geodesic fibers. Therefore, it is straightforward to show that
\[\Delta|_{\Pi^*(C^\infty(\Bbb{CP}^n))}=\Pi^*\circ\Delta_{\Bbb{CP}^n}\circ(\Pi^*)^{-1},\]
where $(\Pi^*)^{-1}$ is the inverse of $\Pi^*:C^\infty(\Bbb{CP}^n)\to\Pi^*(C^\infty(\Bbb{CP}^n))$. As eigenfunctions of $\Delta$ in $C^\infty(S^{2n+1})$ of eigenvalue at most $4t(t+n)$ are in $P_{2t}(S^{2n+1})$, this shows that $\Pi^*(R_t(\Bbb{CP}^n))\subset P_{2t}(S^{2n+1})$, so $R_t(\Bbb{CP}^n)\subset P_{t}(\Bbb{CP}^n)$. To show the reverse inclusion, observe from the well-known decomposition \cite[Theorem 5.7]{Axler...01} of the space of homogeneous polynomials of fixed degree into a direct sum of spaces of homogeneous harmonic polynomials that for any $f\in P_{2t}(S^{2n+1})$, there exist eigenfunctions $f_i\in P_{i}(S^{2n+1})$ of $\Delta$ with eigenvalue $i(i+2n)$ (or zero) which are homogeneous of degree $i$ (if nonvanishing) for $i\in\{0,\dots,2t\}$ such that $f=\sum_{i=0}^{2t}f_i$. Then, we see that if $f\in\Pi^*(P_t(\Bbb{CP}^n))$, we must have $f_i\in\Pi^*(P_t(\Bbb{CP}^n))$, so each $f_i$ is an eigenfunction of $\Pi^*\circ\Delta_{\Bbb{CP}^n}\circ(\Pi^*)^{-1}$ with eigenvalue at most $4t(t+n)$. Therefore, we have $f\in\Pi^*(R_t(\Bbb{CP}^n))$, proving that $P_t(\Bbb{CP}^n)\subset R_t(\Bbb{CP}^n)$. This completes the proof that $P_t(\Bbb{CP}^n)=R_t(\Bbb{CP}^n)$
\end{proof}

\section{Weighted $t$-design curves}
\label{sec:weight}

With $P_t(S^d)$ and $\sigma$ as in Section \ref{sec:approxprelim}, we now formally introduce weighted $t$-design curves:

\begin{definition}\label{def:weight}
Consider $d\in\Bbb N_+$ and $t\in\Bbb N$. For a continuous, piecewise smooth, closed curve $\gamma:[0,1]\to S^d$ with finitely many self-intersections and a function $w:[0,1]\to\Bbb R$, we say that the pair $(\gamma,w)$ is a \emph{weighted $t$-design curve} if, for all $f\in P_t(S^d)$, we have
\[\int_0^1f(\gamma(s))w(s)|\gamma'(s)|\,ds=\int_{S^d}f\,d\sigma.\]
\end{definition}

So, we can observe that such $(\gamma,w)$ will be an (unweighted) $t$-design curve exactly when $w$ is the constant function $1/\ell(\gamma)$. All weighted $t$-design curves of interest in this manuscript have weight function $w$ which is strictly positive at all but finitely many points, so they can be reparameterized to satisfy $w=1/|\gamma'|$. Note that, reparameterizing a weighted $t$-design curve as such, the resulting curve will no longer be piecewise smooth if its speed $|\gamma'|$ is unbounded. For convenience, we will nevertheless refer to any pair $(\gamma,w)$ with unbounded speed which can be reparameterized to produce a weighted $t$-design curve as, itself, a weighted $t$-design curve.

Weighted $t$-design curves are the natural curve-based analogue of weighted $t$-design sets (i.e., quadrature and cubature formulas), which are pairs $(X,\lambda)$ of a finite subset $X\subset S^d$ and a function $\lambda:X\to\Bbb R\setminus\{0\}$ such that
\[\sum_{x\in X}f(x)\lambda(x)=\int_{S^d}f\,d\sigma\]
for all $f\in P_t(S^d)$. In Subsection \ref{sub:buildweight}, we present and prove Proposition \ref{pro:weightconst}, which produces a weighted $t$-design curve on $S^d$ from a weighted $t$-design set on $S^{d-1}$. We then combine this proposition with asymptotic results in strength $t$ concerning weighted $t$-design sets (alongside a construction presented in Subsection \ref{sub:lifting}) to prove Theorem \ref{thm:weight} in Subsection \ref{sub:weightasymp}, then use the proposition combined with asymptotic results in dimension $d$ concerning weighted $t$-design sets to prove a corresponding result for weighted $t$-design curves in Subsection \ref{sub:weightdimasymp}.

\subsection{Building weighted $t$-design curves}
\label{sub:buildweight}

Consider any weighted $t$-design set $(X,\lambda)$ on $S^{d-1}$ of even size $|X|\in2\Bbb N_+$. Writing $X=\{x_i\}_{i=1}^{2N}$, $\Lambda_0:=0$, and
\[\Lambda_i:=\sum_{j=1}^{i}|\lambda(x_j)|/\sum_{j=1}^{2N}|\lambda(x_j)|\quad(i\in\{1,\dots,2N\}),\]
for $s\in[\Lambda_{i-1},\Lambda_i]$, we set
\[\alpha_{X}(s):=\sqrt{1-\alpha_\lambda(s)^2}x_i,\quad\alpha_\lambda(s):=(-1)^{i-1}\left(2\frac{s-\Lambda_{i-1}}{\Lambda_i-\Lambda_{i-1}}-1\right)\]
alongside $l(s):=\lambda(x_i)/(\Lambda_i-\Lambda_{i-1})$. We then consider the pair
\begin{equation}\label{eq:weightconst}
\begin{gathered}
\alpha_{X,\lambda}:=(\alpha_X,\alpha_\lambda):[0,1]\to S^d, \quad
w_{X,\lambda}:=\frac{2cl(1-\alpha_\lambda^2)^{\frac{d-2}2}}{|\alpha'_{X,\lambda}|}:[0,1]\to\Bbb R,
\end{gathered}
\end{equation}
where $c$ is the quotient of the volume of $S^{d-1}$ by the volume of $S^{d}$ (each equipped with their standard measures). %When $\lambda$ is the constant function $1/|X|$ (i.e., when $X$ is an (unweighted) spherical $t$-design), we write $\alpha_{X}:=\alpha_{X,\lambda}$ and $w_X:=w_{X,\lambda}$ for brevity. 

\begin{figure}
\begin{center}
\includegraphics[width=.6\textwidth]
{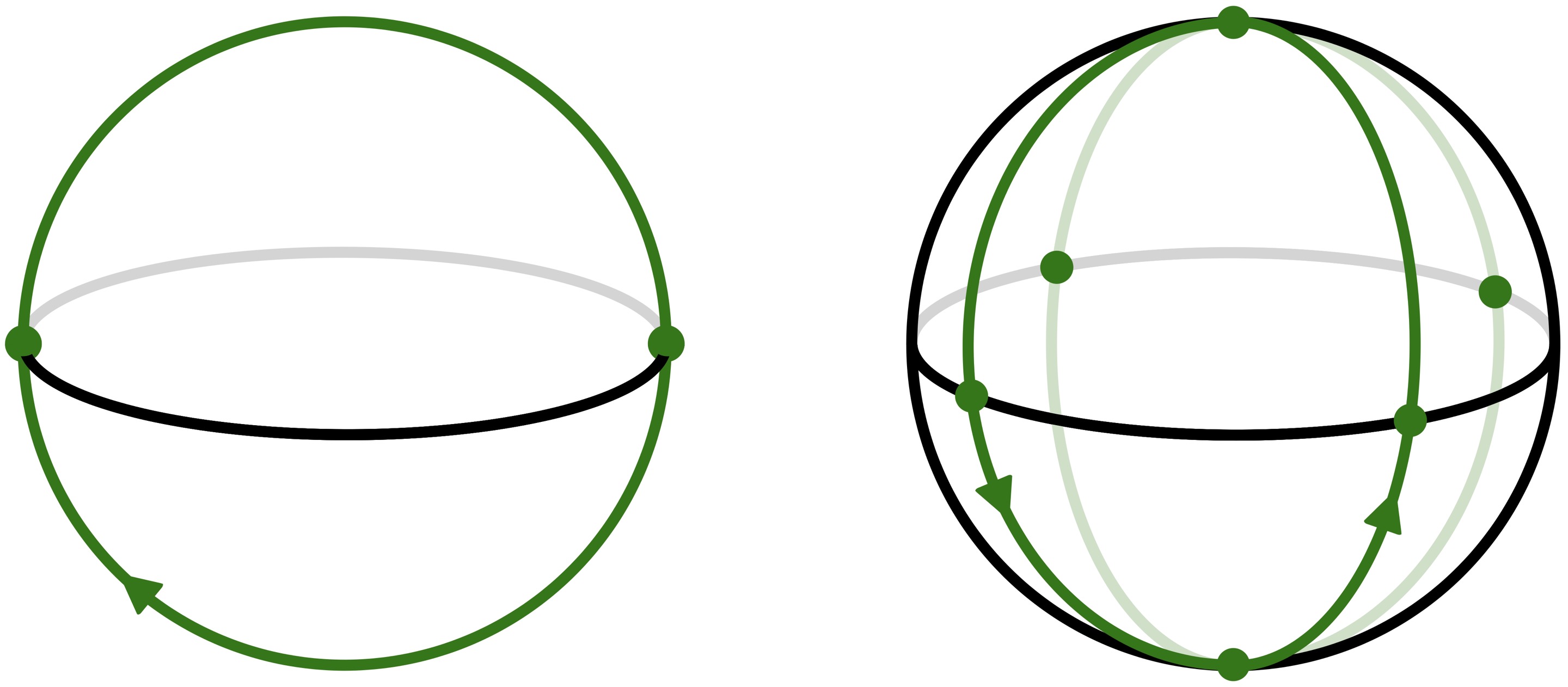}
\caption{\label{fig:Weightex}
Images of the weighted 1-design curve $(\alpha_{V_2,1/2},w_{V_2,1/2})$ (left) and the weighted 3-design curve $(\alpha_{V_4,1/4},w_{V_4,1/4})$ (right) on $S^2$ respectively resulting from the construction of Proposition \ref{pro:weightconst} applied to a 1-design set (antipodal points $V_2$) and 3-design set (the vertices $V_4$ of a 4-gon) on $S^1$.
}
\end{center}
\end{figure}

\begin{proposition}\label{pro:weightconst}
For $d\geq2$, $t\in\Bbb N$, and a weighted $t$-design set $(X,\lambda)$ on $S^{d-1}$ of even size $|X|$, we have that $(\alpha_{X,\lambda},w_{X,\lambda})$ is a weighted $t$-design curve on $S^d$.
\end{proposition} 

\begin{proof}
Consider notation as in the proposition statement alongside the map $h:S^d\to[-1,1]$ taking an element of $S^d\subset\Bbb R^{d+1}$ to its last real coordinate, whose preimages $h^{-1}(s')$ we equip with the uniform spherical measure $\sigma$, normalized so that $\sigma(h^{-1}(s'))=1$. Picking $f\in P_t(S^d)$, note that $f|_{h^{-1}(s')}$ is a polynomial of degree at most $t$ on $h^{-1}(s')\cong S^{d-1}$ for all $s'\in(-1,1)$. We therefore see that
\begin{equation*}
\begin{split}
\int_0^1f(\alpha_{X,\lambda}(s))&w_{X,\lambda}(s)|\alpha_{X,\lambda}'(s)|\,ds \\
&=\sum_{i=1}^{2N}\int_{\Lambda_{i-1}}^{\Lambda_i}2f(\alpha_{X,\lambda}(s))cl(s)(1-\alpha_\lambda(s)^2)^{\frac{d-2}2}\,ds \\
&=\sum_{i=1}^{2N}\int_{-1}^1f(\alpha_{X,\lambda}(\alpha_\lambda|_{[\Lambda_{i-1},\Lambda_i]}^{-1}(s)))c\lambda(x_i)(1-s^2)^{\frac{d-2}2}\,ds \\
&=\int_{-1}^1\sum_{i=1}^{2N}\lambda(x_i)f\left(\sqrt{1-s^2}x_i,s\right)c(1-s^2)^{\frac{d-2}2}\,ds \\
&=\int_{-1}^{1}\int_{h^{-1}(s)}f(\omega)d\sigma(\omega)c(1-s^2)^{\frac{d-2}2}\,ds \\
&=\int_{S^d}f(\omega)\,d\sigma(\omega),
\end{split}
\end{equation*}
so $(\alpha_{X,\lambda},w_{X,\lambda})$ is a weighted $t$-design curve, as desired.
\end{proof}

\begin{figure}
\begin{center}
\includegraphics[width=.75\textwidth]
{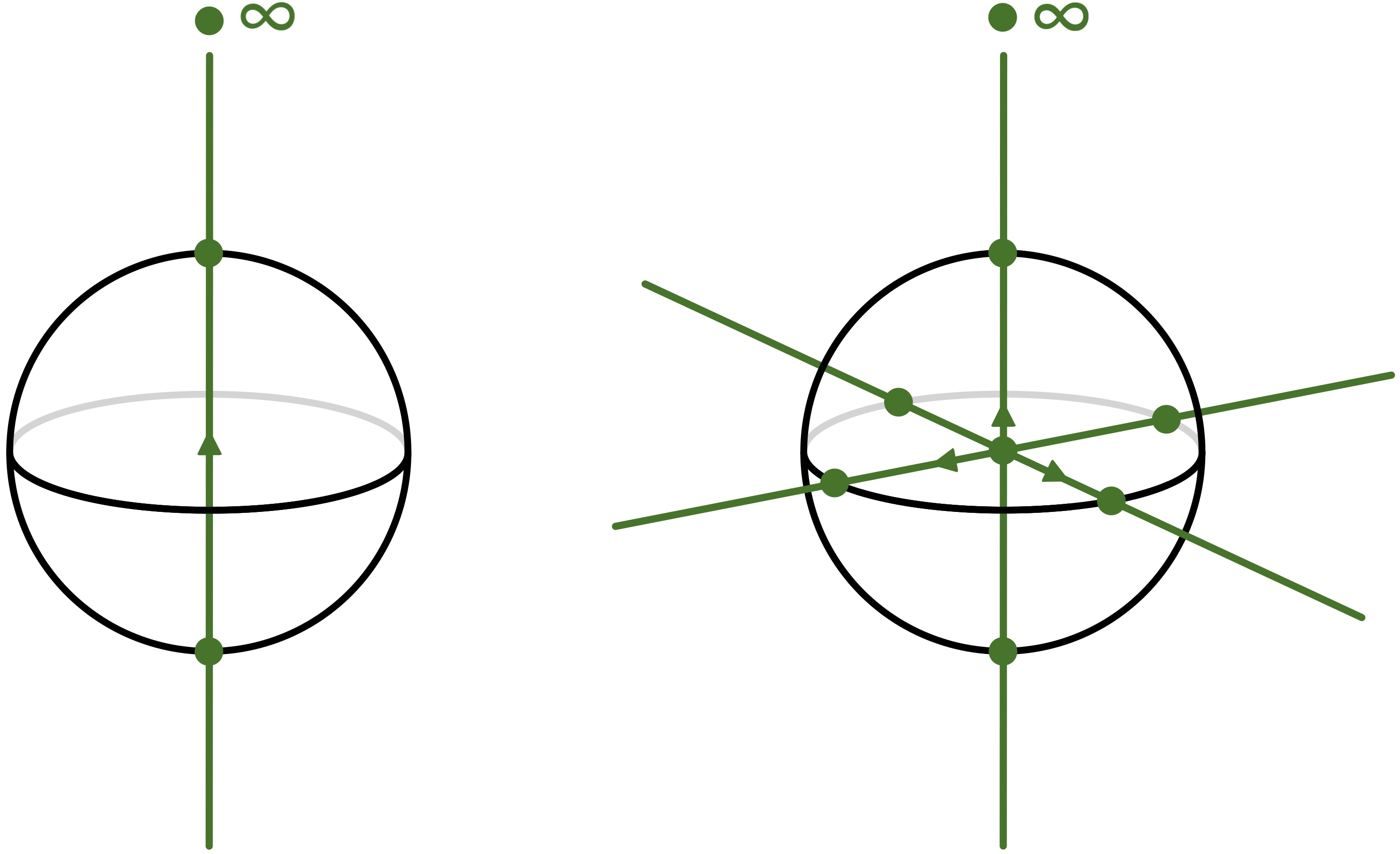}
\caption{\label{fig:Weightex2}
Images of a weighted 1-design curve (left) and a weighted 3-design curve (right) on $S^3\cong\Bbb R^3\cup\{\infty\}$ respectively resulting from the construction of Proposition \ref{pro:weightconst} applied to a 1-design set (antipodal points) and 3-design set (the vertices of an octahedron) on $S^2$.
}
\end{center}
\end{figure}

In fact, for any piecewise smooth map $R:[-1,1]\to\text{SO}(d)$, defining
\begin{equation}\label{eq:rotation}
\begin{gathered}
\alpha_{X,\lambda,R}:=((R\circ\alpha_{\lambda})\alpha_X,\alpha_{\lambda}),\quad w_{X,\lambda,R}:=\frac{2cl(1-\alpha_{\lambda}^2)^{\frac{d-2}2}}{|\alpha'_{X,\lambda,R}|}
\end{gathered}
\end{equation}
as in \eqref{eq:weightconst}, the proof of Proposition \ref{pro:weightconst} directly shows that $(\alpha_{X,\lambda,R},w_{X,\lambda,R})$ is also a weighted $t$-design curve.

\subsection{Lifting weighted $t$-design curves}
\label{sub:lifting}

The construction \cite[Theorem 1.3]{Lindblad24b} of the present author which builds a $t$-design curve on $S^3$ from a $\lfloor t/2\rfloor$-design curve on $S^2$ directly generalizes to the setting of weighted $t$-design curves to provide the result of Theorem \ref{thm:S3stitch}:

\begin{theorem}[Weighted analogue of Theorem 1.3 in work \cite{Lindblad24b} of the present author]\label{thm:S3stitch}
For any $t\in\Bbb N$, any weighted $\lfloor t/2\rfloor$-design curve $(\alpha=(\alpha_\Bbb R,\alpha_\Bbb C),w_\alpha)$ on $S^2\subset\Bbb R\times\Bbb C$ satisfying $\alpha_\Bbb R(s)\neq-1$ for $s\in[0,1]$, any positive integer $g$ coprime to $t+1$, and any continuous map $\theta:[0,1]\to\Bbb R$ which is smooth on the complement of a finite subset of $[0,1]$, whose derivative is $L^1$-integrable, and which satisfies $\theta(1)-\theta(0)\in2\pi\Bbb Z$, we define
\begin{equation*}\label{eq:lift}
\begin{gathered}
\gamma_{\alpha,\theta}:[0,1]\to S^3,\quad s\mapsto\frac1{\sqrt2}\left(\sqrt{1+\alpha_\Bbb R(r)}, \frac{\overline{\alpha_\Bbb C(r)}}{\sqrt{1+\alpha_\Bbb R(r)}}\right)e^{2\pi isg+i\theta(r)}, \\
w_{\alpha,\theta}:[0,1]\to \Bbb R,\quad s\mapsto\frac{w_\alpha(r)|\alpha'(r)|}{|\gamma_{\alpha,\theta}'(s)|} \\
\end{gathered}
\end{equation*}
for $r:=(t+1)s-\lfloor(t+1)s\rfloor$. $(\gamma_{\alpha,
\theta},w_{\alpha,\theta})$ is a weighted $t$-design curve on $S^3$.
\end{theorem}

The proof of Theorem \ref{thm:S3stitch} follows as in the unweighted setting \cite[Theorem 1.3]{Lindblad24b}, so we refer to the argument in that setting to verify the theorem. Also as in this unweighted setting, for any $\varepsilon>0$ and $\theta_1$ as in the theorem statement, we may perturb $\theta_1$ to some $\theta_2$ such that $\gamma_{\alpha,\theta_2}$ is simple and $\ell(\gamma_{\alpha,\theta_2})=\ell(\gamma_{\alpha,\theta_1})+\varepsilon$. When $\gamma_{\alpha,\theta_1}$ is an unweighted $t$-design curve, we may then arrange that the same is true of $\gamma_{\alpha,\theta_2}$. Note also that when $|\alpha'|$ is bounded on the subset of $[0,1]$ on which $\alpha$ is smooth and $w_\alpha$ is positive and bounded away from 0 and $\infty$, we may pick $\theta$ such that $\gamma_{\alpha,\theta}$ is an unweighted $t$-design curve  by varying $\theta$ such that $w_{\alpha,\theta}$ is constant for all $s\in[0,1]$ at which it is defined.

\subsection{Asymptotics of weighted $t$-design curves}
\label{sub:weightasymp}

Fix any $t\in\Bbb N_+$. Noting that the vertices 
\begin{equation}\label{eq:verts}
V_{2t}:=\{e^{\pi i(2j+1)/(2t)}\in S^1\:|\:j\in\{0,\dots,2t-1\}\}
\end{equation}
of a $2t$-gon are a $(2t-1)$-design set on $S^1$ for all $t\in\Bbb N$ \cite[Example 5.14]{Delsarte...77}, we can see from Proposition \ref{pro:weightconst} that $(\alpha_{V_{2t},1/(2t)},w_{V_{2t},1/(2t)})$ as in \eqref{eq:weightconst} is a weighted $(2t-1)$-design curve on $S^2$ of length $2\pi t$. We may then observe that this pair gives exactly $(\alpha,1/|\alpha'|)$ as in \eqref{eq:explicitweight2}. Moreover, 
setting $\alpha:=(\alpha_\Bbb R,\alpha_\Bbb C):=\alpha_{V_{2t},1/(2t)}$ and $w_\alpha:=w_{V_{2t},1/(2t)}$, we then have from \eqref{eq:verts} that $\alpha_\Bbb R(s)\neq-1$ for all $s\in[0,1]$. Considering $\theta:[0,1]\to\Bbb R$ as in Theorem \ref{thm:S3stitch} and setting $g:=2t+1$ (which we note will always be coprime to $4t$), we may then apply Theorem \ref{thm:S3stitch} to produce a weighted $(4t-1)$-design curve $(\gamma_{\alpha,\theta},w_{\alpha,\theta})$ on $S^3$. We may observe from the proof of the unweighted analogue \cite[Theorem 1.3]{Lindblad24b} of Theorem \ref{thm:S3stitch} that the length computation \cite[(1)]{Lindblad24b} holds in that setting; therefore, the length of $\gamma_{\alpha,\theta}$ will be greater than or equal to $2\pi\sqrt{4t^4+1}$, with equality for some $\theta$. We may also then observe that after perturbation of any given $\theta$, we may arrange that $\gamma_{\alpha,\theta}$ is simple. We can then see that $(\gamma_{\alpha,\theta},w_{\alpha,\theta})$ is exactly $(\gamma,1/|\gamma'|)$ as in \eqref{eq:explicitweight3}.

We now verify Theorem \ref{thm:weight} for general $d\geq2$. Taking $W_d\geq\pi C_{d-1}$ such that $W_{d}/\pi$ is even for $C_{d-1}$ as in work of Bondarenko, Radchenko, and Viazovska \cite[Theorem 1]{Bondarenko...13}, for any $t\in\Bbb N_+$, there then exists a $t$-design set $X$ on $S^{d-1}$ of even size $W_{d}t^{d-1}/\pi$. Proposition \ref{pro:weightconst} then gives that $(\alpha_{X,1/|X|},w_{X,1/|X|})$ is a weighted $t$-design curve on $S^d$, and we may observe that $\alpha_{X,1/|X|}$ has length $W_{d}t^{d-1}$ and that $w_{X,1/|X|}$ is strictly positive at all but finitely many points. 
For any $W\geq W_{d}t^{d-1}$, we may then consider $R$ and $(\alpha_{X,1/|X|,R},w_{X,1/|X|,R})$ as in \eqref{eq:rotation} such that the length of $\alpha_{X,1/|X|,R}$ is exactly $W$, and we may observe that $w_{X,1/|X|,R}$ will still be strictly positive at all but finitely many points. This completes the proof of Theorem \ref{thm:weight}. We then see from Proposition \ref{pro:weightEG} (which follows exactly as in the proof of Theorem 1.1 of Ehler and Gr\"ochenig \cite{EhlerGrochenig23}) that the asymptotic order of length of the curves $(\alpha_{X,1/|X|},w_{X,1/|X|})$ is optimal among weighted $t$-design curves with non-negative weight functions:

\begin{proposition}\label{pro:weightEG}
For each $d\in\Bbb N_+$, there exists a constant $w_d>0$ such that for any $t\in\Bbb N_+$, the length of any weighted $t$-design curve on $S^d$ with non-negative weight function is at least $w_dt^{d-1}$.
\end{proposition}

\subsection{Asymptotics in dimension of weighted $t$-design curves}
\label{sub:weightdimasymp}

Work \cite[Theorem 1.6]{Dillon24} of Dillon shows that for each $t\in\Bbb N_+$, there exists a constant $D_t$ such that for any $d-1\in\Bbb N_+$, there exists a weighted $t$-design set $(X,\lambda)$ on $S^{d-1}$ such that $\lambda$ is strictly positive and $X$ has size (which we may assume is even after potentially replacing the design with two copies of itself, one generically rotated away from the other) at most $D_td^{t-1}$. Picking $\widetilde W_{t}\geq \pi D_t$ such that $\widetilde W_t/\pi$ is an even integer, for any $\widetilde W\geq\widetilde W_td^{t-1}$, we may then pick smooth $R:[-1,1]\to\text{SO}(d)$ such that $(\alpha_{X,\lambda,R},w_{X,\lambda,R})$ as in \eqref{eq:rotation} is a weighted $t$-design curve on $S^d$ of length $\widetilde W$, and we may observe that $w_{X,\lambda,R}$ is strictly positive at all but finitely many points. This proves Proposition \ref{pro:weightstrength}, which---in analogy with results concerning the asymptotic order of size of weighted $t$-design sets on spheres $S^d$ as $d\to\infty$ for fixed $t$ \cite{Tchakaloff57,Dillon24}---gives a result concerning the asymptotic order of length of weighted $t$-design curves as $d\to\infty$ for fixed $t$.

\begin{proposition}\label{pro:weightstrength}
For any $t\in\Bbb N_+$, there exists a constant $\widetilde W_t$ such that for any $d\geq2$ and $\widetilde W\geq\widetilde W_td^{t-1}$, there exists a weighted $t$-design curve on $S^d$ of length $\widetilde W$ with weight function strictly positive except at finitely many points.
\end{proposition}

The weighted $t$-design curves said to exist in Proposition \ref{pro:weightstrength} are not in general \emph{asymptotically optimal} as $d\to\infty$: we may produce weighted 2-design and 3-design curves on $S^d$ of lengths $\pi(d+1)$ and $2\pi d$ respectively by applying Proposition \ref{pro:weightconst} to the vertices of a regular simplex (for $d$ odd) and of a cross-polytope on $S^{d-1}$, which respectively constitute 2-design and 3-design sets.

\section*{Acknowledgements}
\label{sec:thanks}

The author would like to thank Henry Cohn, Tom Mrowka, and Karlheinz Gr\"ochenig for helpful comments and discussions. The author would also like to thank the School of Science, the Department of Mathematics, and the Office of Graduate Education for support through the MIT Dean of Science fellowship during their doctoral studies, as well as NSF grant DMS-2105512 and the Simons Foundation Award \#994330 (Simons Collaboration on New Structures in Low-Dimensional Topology).

\bibliography{Curvesbib.bib}

\end{document}